\documentclass[draft]{amsart}

\usepackage{amssymb}
\usepackage{url}
\theoremstyle{plain}
\newtheorem{thm}{Theorem}[section]
\newtheorem{cor}[thm]{Corollary}
\newtheorem{lem}[thm]{Lemma}
\newtheorem{prop}[thm]{Proposition}
\newtheorem{exam}[thm]{Example}
\newtheorem{rem}[thm]{Remark}
\newtheorem{Def}[thm]{Definition}

\begin{document}

%%%%%%%%%%%%%%%%%%%%%%%%%
% Subject classification
%%%%%%%%%%%%%%%%%%%%%%%%%

% Provide an AMS subject classification with one or two primary classification
% numbers and, optionally, one or more secondary classification numbers.
% Use the following format:  "Primary 42B25. Secondary 42B60, 20E26"


%%%%%%%%%
% Title
%%%%%%%%%

% Title, in lower case, with no explicit linebreaks (\\).  If the title
% is too long to be used as a running head, add a short version of the
% title in brackets, as in \title[shorttitle]{fulltitle}.

\title{Operator analogues of Mahler's measure}

%%%%%%%%%%%%%%%%%%%%%%%%%%%%%%
% Author names and addresses
%%%%%%%%%%%%%%%%%%%%%%%%%%%%%%

% Provide one separate \author{...} \address{...} \email{....} entry for each
% author, i.e., do not combine multiple authors.  Separate address lines by double
% slashes.  Do not attach footnotes to author  names. (For acknowledgements use
% the "\thanks" construct below.)
%

\author{Kunyu Guo}
\address{ School of Mathematical Sciences
Fudan University, 200433-Shanghai, China
}

\email{kyguo@fudan.edu.cn}

\author{Jiayang Yu}
\address{ School of Mathematical Sciences
Fudan University, 200433-Shanghai, China
}

\email{j.yu520986@gmail.com}
%%%%%%%%%%%%%%%%%%%%
% Acknowledgements
%%%%%%%%%%%%%%%%%%%

% Use \thanks for acknowledgements as footnotes to the title page.
% (Note that footnotes inside \author or \title macros are not
% allowed.)
%
% In case of multiple author papers, phrase the acknowledgement to
% say "The first author was supported by ...  The second author was
% supported by ..."

\thanks{This work is partially supported by Laboratory of Mathematics for Nonlinear Science, Fudan University. The second author was partially supported by NSFC ( 11271075).}

%%%%%%%%%%%%%
% Abstract
%%%%%%%%%%%%%
%
% Abstracts should not contain macros (so that they can be processed independently
% of the paper.) Avoid displayed math and references in the abstract.

\begin{abstract}
Motivated by a geometric meaning of Mahler's
measure, we introduce two operator analogues of Mahler's measure. This leads to  some interesting equalities and inequalities between
the two operator-theoretic Mahler measures and the classical Mahler measure.
In order to apply these results to the operator version of Lehmer's problem, we introduce and study an important class of operators, the so-called  subharmonic operators.
 It is shown that the operator version of Lehmer's problem fails
under some mild condition.
\end{abstract}

\maketitle

%%%%%%%%%%%%%%%%%%%%%%%%%%%%%%%%%%%%%%%%%%%%%%%%%%%%%%%%%%%%%%%%%%%%%%%%%
% end Topmatter
%%%%%%%%%%%%%%%%%%%%%%%%%%%%%%%%%%%%%%%%%%%%%%%%%%%%%%%%%%%%%%%%%%%%%%%%%

\section{Introduction}

 Let $\mathbb{Z}[z]$ and $\mathbb{C}[z]$ denote  the polynomial rings in $z$ with integer and complex
 coefficients, respectively. Denote the open unit disk by
 $\mathbb{D}$, and the unit circle by $\mathbb{T}$. In this paper, $H $ always denotes a Hilbert space, and $B(H)$ denotes the set of
  all  linear bounded operators acting on  $H$.

In order to manufacture large primes, Lehmer paid his attention to
monic integral polynomial
 $$
 p(z)=z^d+a_{d-1}z^{d-1}+\ldots+a_1 z+a_0\in \mathbb{Z}[z].
 $$
Decompose $p(z)$ on $\mathbb{C}$ as
 $$
  p(z)=\prod_{i=1}^{d}(z-\alpha_i),
 $$
 and   define
 $$
 \Delta_n(p)=\prod_{i=1}^{d}(\alpha_i^n-1), \, n= 1,2,\cdots.
 $$
 Since $p(z)$ is a monic integral polynomial, it is easy to see that $\Delta_n(p)\in \mathbb{Z}$. The function $\Delta_n(p)$ was introduced by  Pierce \cite{Pi}.
 In 1933,  Lehmer \cite{Le} proved that $\Delta_n(p)$ is more likely to produce primes if it does not grow too quickly.
 Let $\Omega(p)$ be the absolute value of the product of those roots of $p$ which lie outside the unit circle.
 If $p$ has no root on the unit circle, then $\lim\limits_{n\to \infty}|\frac{\Delta_{n+1}(p)}{\Delta_{n}(p)}|=\Omega(p)$.
 Thus for any monic integral polynomial $p$, Lehmer used $\Omega(p)$ to measure the rate of growth of the
 sequence $\{\Delta_{n}(p)\}_{n=1}^{\infty}$.
It is clear that $\Omega(p)\geq 1$. Lehmer noticed the polynomial
  $$
  L(z)=z^{10}+z^9-z^7-z^6-z^5-z^4-z^3+z+1
  $$
 with $\Omega(L)=1.176280\cdots$. However, he
  failed to find a monic integral polynomial $p$ such that $1<\Omega(p)<\Omega(L)$.
Then he asked if for every $\epsilon>0$ there exists a monic
polynomial $p\in\mathbb{Z}[z]$ satisfying
 $1<\Omega(p)<1+\epsilon$. This is known as``Lehmer's problem" or ``Lehmer's conjecture",
 which remains to be an open problem.

 Thirty years after Lehmer's paper \cite{Le}, Mahler gave
 a generalized definition of $\Omega(p)$. For a  nonzero polynomial
 $$
  p(z)=a_d z^d+a_{d-1}z^{d-1}+\ldots+a_0=a_d\prod_{i=1}^{d}(z-\alpha_i)\in \mathbb{C}[z],
 $$
 he defined
 $$
 M(p)=|a_d|\cdot \prod_{i=1}^{d}\max \{1,|\alpha_i|\}.
 $$
 $M(p)$ is called the Mahler measure of $p$.
Observe that $M(p)\geq 1$ for each $p\in\mathbb{Z}[z]$ and when
$a_d=1,\ M(p)=\Omega(p)$. By a classical theorem of Kronecker
\cite[p.27, Theorem 1.31]{EW} \cite{Kr}, for any $p\in\mathbb{Z}[z]$,
  $M(p)=1$ if and only if
    $p(z)=z^n q(z)$ for some nonnegative integer $n$ and a
    cyclotomic polynomial $q$. Recall that a {\it cyclotomic
    polynomial} is a monic integral polynomial all of whose zeros
    are roots of unity. Thus Lehmer's problem is equivalent to the
    question: Is there a sequence of non-cyclotomic
integral polynomial $\{p_n\}$ with $p_n(0)\neq0$ for all $n$ such that
\begin{eqnarray}
\lim_{n\to
\infty} M(p_n) =1?\label{q1.1}
\end{eqnarray}

Lehmer's problem and Mahler's measure arise in different areas of
mathematics, for example, iteration of complex functions,
transcendence and diophantine approximation theory, Fuglede-Kadison
determinant in operator algebra \cite{De}, ergodic theory \cite{Li}, knot theory \cite{Hi}, and
etc. See \cite{Sm} for survey of Lehmer's problem, and also refer to
 \cite{Mo}.

An important observation was also made in Mahler's paper
\cite{Ma60}:

$$
M(p)=\exp\bigg[\frac{1}{2\pi}\int_{-\pi}^{\pi}\log |p(e^{i\theta
})|\,\mathrm{d}\theta\bigg]
,\ \ \ \ \  p\,  \in\mathbb{C}[z].
$$
This means that $M(p)$ is multiplicative. Combining  Szeg\"o 's
theorem, one can give a geometric meaning of Mahler's measure in the
context of Hilbert space, that is,
\begin{eqnarray}
M(p)=\text{dist}(p,[z
p])=\inf_{q\in\mathbb{C}[z]}\|[1-q(z)z] p(z) \|_{H^{2}}, \ \,
 p  \in\mathbb{C}[z].\label{1.1}%(\text{see section 2})
\end{eqnarray}
Here $$H^{2}=\{f: f\in L^2(\mathbb{T}), \widehat{f}(n)=0\text{ for
all }n<0\}$$ is the Hardy space on the unit circle $\mathbb{T}$ and $[z p]$ denotes the closed invariant subspace of the Hardy
shift generated by $zp$, that is,
$$[z p]=\text{cl}\{q(z)zp(z),q\in \mathbb{C}[z]\}.$$ In (\ref{1.1}), the equality
$M(p)=\inf_{q\in\mathbb{C}[z]}\|[1-q(z)z ]p(z) \|_{H^{2}}$ was
known   in \cite{Deg}.
%but we don't see the form $M(p)=\text{dist}_{H^{2}}(p,[z p])$ before.

Let  $S$ denote the Hardy shift, i.e.
 \begin{eqnarray*}
(Sf)(z)=zf(z),\ \ f\in H^2.
\end{eqnarray*}
Then we can rewrite (\ref{1.1}) as follows:
\begin{eqnarray*}
M(p)=\text{dist}(p(S)1,[Sp(S)1])&=&\inf_{q\in\mathbb{C}[z]}\|[I-q(S)S]
p(S)1\|.
\end{eqnarray*}
By using the inner-outer factorization of functions in $H^2$, one
will see that for any unit vector $e$ in $H^2$
$$
 \text{dist}(p(S)e,[Sp(S)e]) \leq M(p), \ \  p\,  \in\mathbb{C}[z](\text{ see section 2 for details }).
$$
Hence
$$
M(p)=\sup_{\|e\|=1}\text{dist}(p(S)e,[Sp(S)e]).
$$

Inspired by these observations, we will introduce  and study two operator analogues
of Mahler's measure.

Let $T\in B(H)$.
 For $h\in H$, let $[h]$ denote the closed invariant subspace  of $T$ generated by $h$, that is, $$[h]=\overline{\text{span}}\{h,
 Th,T^2h,\ldots\}.$$
 Then  for each  polynomial  $p$ and $e\in H$, define
$$
M_T^e(p)=
\text{dist}(p(T)e,[Tp(T)e])=\inf_{q\in\mathbb{C}[z]}\|[I-q(T)T]
p(T)e \|,
$$ 
called \emph{ the $T$-Mahler measure of  $p$ on $e$}; and set
$$
M_T(p) =  \sup_{\|e\|=1}M_T^e(p),
$$
called \emph{the $T$-Mahler measure of   }$p$.
We will establish some connection between the $T$-Mahler measure and the classical Mahler measure. This makes it possible to study Lehmer's problem in the context of operator theory.

This paper is organized as follows.

In Section 2, we will pay attention to the properties of the $T$-Mahler measure.
In particular, when $T$ is a contraction,
$M_T(p)\leq M(p)$ for all $p\in \mathbb{C}[z]$. The multiplicativity  properties of $M_T$ and $M_T^e$
are also studied.

In Section 3, in order to generalize Lehmer's problem in the context of operator theory, we will introduce and study an important class of operators, the so-called  subharmonic operators which is closely related to the operator-theoretic Mahler measure.

Section 4 is devoted to applications of our results in previous sections. It is shown that the operator version of Lehmer's problem fails
 under some mild condition.
As an application, one
gives new proofs of some known results in \cite{Pr08a} and \cite{Hu}, see
Example \ref{e4.5} and Remark \ref{r4.7}.

\section{ Operator analogues of Mahler's measure}
In this section,  we present a geometric meaning of Mahler's measure.

Motivated by this, two operator analogues of Mahler's measure  are introduced.
Some connection between the $T$-Mahler measure and the classical Mahler measure are realized.
Finally, we will pay attention to the multiplicativity properties of the $T$-Mahler measure.

Let us recall the classical Szeg\"o theorem \cite[p.49]{Ho}.
Let $\mu$ be a finite positive Borel measure on the unit circle $\mathbb{T}$ and $h$ be the derivative of $\mu$ with
respect to the normalized Lebesgue measure.
That is,
$$\mathrm{d}\mu=h(e^{i\theta })\frac{\mathrm{d}\theta}{2\pi}+\mathrm{d}\mu_s,$$
where $\mathrm{d}\mu_s$ and $\frac{\mathrm{d}\theta}{2\pi}$ are mutually singular and $h\in L^1(\frac{\mathrm{d}\theta}{2\pi})$.
Szeg\"o 's theorem states that %asserts that
$$
\inf_{f \in A_0}\int|1-f|^2 \,\mathrm{d}\mu\,=\,\exp\bigg[\frac{1}{2\pi}\int_{-\pi}^{\pi}\log h(e^{i\theta })\,
\mathrm{d}\theta\bigg],
$$
where $A_0$ is the disk algebra defined by
$$\{f: f\in C(\mathbb{T}),\hat{f}(n)\triangleq \int_{0}^{2\pi}f(e^{ i \theta} ) e^{-in\theta} \frac{d\theta}{2\pi}=0, n\leq 0\}.$$
In particular, if $p \in \mathbb{C}[z]$ and
  $\mathrm{d}\mu=|p|^2\frac{\mathrm{d}\theta}{2\pi} $, then by  Szeg\"o 's theorem we have
\begin{eqnarray}
M^2 (p)&=&\exp\bigg[\frac{1}{2\pi}\int_{-\pi}^{\pi}\log |p|^2\,\mathrm{d}\theta\bigg]\nonumber \\
&=&\inf_{f \in A_0}\int_{T}|1-f|^2|p|^2 \,\frac{\mathrm{d}\theta}{2\pi} \nonumber \\
&=&\inf_{q \in \mathbb{C}[z]}\int_{T}|(1-qz)p|^2 \,\frac{\mathrm{d}\theta}{2\pi} \nonumber \\
&=& \inf_{q \in \mathbb{C}[z]}\|(1-qz)p\|^2 \label{e2.1} \\
&=& \text{dist}^2(p,[zp]). \label{e2.2}
 \end{eqnarray}
  As mentioned in the introduction, let  $S$ be the Hardy shift. Then we have the following operator-theoretic form of Mahler's measure
\begin{eqnarray}
M(p)=\text{dist}(p,[z p])&=&\inf_{q\in\mathbb{C}[z]}\|[1-q(z)z ]p(z) \|  \nonumber  \\
&=&\inf_{q\in\mathbb{C}[z]}\|[I-q(S)S] p(S)1\| \nonumber  \\
&=& \,\text{dist}(p(S)1,[Sp(S)1]). \label{e2.3}
\end{eqnarray}
Inspired by this observation, we have the following definition.
\begin{Def}   Let $T\in B(H)$ and $[h]$ denote the closed invariant subspace  of $T$ generated by $h\in H$.
Then for each vector $e\in H$,  define
\begin{eqnarray*}
M_T^e(p)= \text{\upshape dist}(p(T)e,[Tp(T)e])=\inf_{q\in\mathbb{C}[z]}\|[I-q(T)T] p(T)e \|,\ \ p\in\mathbb{C}[z],
\end{eqnarray*} called  the $T$-Mahler measure on $e$.
\end{Def}
For example, let $T=S$. For each unit vector $e\in H^2$  and any polynomial $p$, we have
\begin{eqnarray*}
M_S^{e}(p)&=&\text{dist}(p(S)e,[Sp(S)e])\\
&=&\Bigg(\inf_{f \in \mathbb{C}[z],f(0)=0}\int_{\mathbb{T}}|pe-fpe|^2 \,\frac{\mathrm{d}\theta}{2\pi}\Bigg)^{\frac{1}{2}} \\
&=&\exp\bigg[\frac{1}{2\pi}\int_{-\pi}^{\pi}\log |pe|\,\mathrm{d}\theta\bigg]\\
&=&\exp\bigg[\frac{1}{2\pi}\int_{-\pi}^{\pi}\log |p|\,\mathrm{d}\theta\bigg] %\cdot
\exp\bigg[\frac{1}{2\pi}\int_{-\pi}^{\pi}\log |e|\,\mathrm{d}\theta\bigg]\\
&\leq&\exp\bigg[\frac{1}{2\pi}\int_{-\pi}^{\pi}\log |p|\,\mathrm{d}\theta\bigg]\\
&=&M(p).
\end{eqnarray*}
Thus
$$
\sup_{\|e\|=1}M_S^e(p)=M(p).
$$
This leads to the following definition.
\begin{Def} Let $T\in B(H),$  set
\begin{eqnarray*}
M_T (p)=\sup_{\|e\|=1}M_T^e(p),\ \ p\in \mathbb{C}[z].
\end{eqnarray*} Then $M_T$ is  called the $T$-Mahler measure.
\end{Def}
%By the  above argument, both $M_S^{1}$ and $M_S$ are equal to the Mahler's measure.
It is easy to verify that $M_T$ is
unitary invariant for $T$. This means that if $T_1$ is unitarily equivalent to $T_2$, then $M_{T_1}=M_{T_2}$.
Also, observe  that if $M_T\neq0$ then $T$ has a nontrivial invariant subspace.

From the discussion before Definition 2.2, one sees that both $M_S^{1}$ and $M_S$ are equal to Mahler's measure.
In the following, we will give more results related to Mahler's measure.

\subsection{Some connection between the $T$-Mahler measure and the classical Mahler measure}

We state our main result in this section as follows, and its proof is placed at the end of Section 2.1.
\begin{thm}
For any contraction $T\in B(H),$ i.e. $\|T\|\leq1$,  we have \label{2.6}
$$
M_{T}(p)\leq M(p) ,\, \  p\in \mathbb{C}[z].
$$
\end{thm}
First, we establish the following lemma.
\begin{lem}
Suppose $V$ is an isometry on $H$. Then \label{2.3}
$$
M_{V}(p)= M_{V}(1)\cdot M(p) ,\, \ p\in  \mathbb{C}[z].
$$
\end{lem}
\begin{proof}
Suppose $V$ is an isometry on $H$. It is well-known that $V$ has a unitary extension $U$, where $U$ can be decomposed as

\[
U=\left[
\begin{array}{ll}
V&I_H-VV^*\\
0&V^*
\end{array}
\right],
\]
with respect to $K \triangleq H\oplus H$, see [Pau, p.6].

For each unit vector $e\in H$, denote by $[e]$ the closed reducing subspace of $U$ generated by $e$,
and by $U_{[e]}$ the restriction of $U$ on $[e]$. Then $U$ can be be decomposed as
\[
U=\left[
\begin{array}{ll}
U_{[e]}&0\\
0&U'
\end{array}
\right],
\]
with respect to $ [e]\oplus [e]^{\perp}$.
It is easy to verify that
$$
M_{V}^{e}(p)=M_{U}^{e}(p)=M_{U_{[e]}}^{e}(p),\ \ p\in \mathbb{C}[z].
$$

 Recall that an operator $T\in B(H)$ is called
\emph{star-cyclic}  if there is a vector $h\in H$ such that
$$H=\text{cl}\{p(T,T^*)h:p \text{ is noncommutative polynomial in two
variables }\}.$$ Observe that $U_{[e]}$ is a
star-cyclic unitary operator.   By the classical theory  of normal operator
[ Co09, p.269; Co00, p.51], there exists a unitary operator
$U_0:[e]\rightarrow L^2(\mu)$ satisfying $$N_{\mu}=U_0U_{[e]}U_0^{-1} \ \  \mathrm{and} \ \ U_0 e=1.$$
Here, $\mu$ is a probability Borel measure on the unit circle $\mathbb{T}$, and  $N_{\mu}$ is the multiplication operator on $L^2(\mu)$ defined by
\begin{eqnarray*}
(N_{\mu}f)(z)=zf(z),\ \ f\in L^2(\mu).
\end{eqnarray*}

Observe that if there is a unitary operator $U_1$ such that
$T_1=U_1T_2 U^{*}_1$ and $U_1e_2=e_1$, then $M_{T_1}^{e_1}=M_{T_2}^{e_2}$. Then  for each $p\in \mathbb{C}[z]$,  we have
\begin{eqnarray*}
M_{U_{[e]}}^{e}(p)&=&M_{N_{\mu}}^{1}(p)\\
&=&\bigg(\inf_{f \in A_0}\int|p-fp|^2 \,\mathrm{d}\mu\bigg)^{\frac{1}{2}}\\
&=&\bigg(\inf_{f \in A_0}\int|1-f|^2 \,|p|^2\mathrm{d}\mu\bigg)^{\frac{1}{2}}\,\,\,\,\,\,\,
\big(\mathrm{d}\mu=h(e^{i\theta})\frac{\mathrm{d}\theta}{2\pi}+\mathrm{d}\mu_s \big)\\
&=&\bigg(\exp\bigg[\frac{1}{2\pi}\int_{-\pi}^{\pi}\log |p(\theta)|^2|h(e^{i\theta})|\,\mathrm{d}\theta\bigg]\bigg)^{\frac{1}{2}} \,\,\,\,\,(\text{ by Szeg\"o 's theorem})\\
&=&\exp\bigg[\frac{1}{4\pi}\int_{-\pi}^{\pi}\log |h(\theta)|\,\mathrm{d}\theta\bigg]\cdot M(p),
\end{eqnarray*}
where
$\mathrm{d}\mu=h(e^{i\theta})\frac{\mathrm{d}\theta}{2\pi}+\mathrm{d}\mu_s$
is the Lebesgue decomposition  of $\mathrm{d}\mu$ relative to
$\frac{\mathrm{d}\theta}{2\pi}$.  %check it???(It is Lebesgue decomposition, see pg 121 of Rudin's real and complex analysis)
Thus
$$
M_{V}^{e}(p)=\exp\bigg[\frac{1}{4\pi}\int_{-\pi}^{\pi}\log |h(\theta)|\,\mathrm{d}\theta\bigg]\cdot M(p),\ \ p\in \mathbb{C}[z].
$$
In particular, we have
$$
M_{V}^{e}(1)=\exp\bigg[\frac{1}{4\pi}\int_{-\pi}^{\pi}\log |h(\theta)|\,\mathrm{d}\theta\bigg],
$$
which implies that
\begin{equation}
 M_{V}^{e}(p)=M_{V}^{e}(1)\cdot M(p),\ \ p\in \mathbb{C}[z].  \label{e2.4}
\end{equation}
 Therefore, for each $p\in \mathbb{C}[z]$  we have
$$
M_{V}(p)=\sup_{\|e\|=1}M_{V}^{e}(p)=\bigg (\sup_{\|e\|=1}M_{V}^{e}(1)\bigg )\cdot M(p)=M_V(1)\cdot M(p).
$$
This completes the proof.
\end{proof}
The following lemma is of independent interest.
\begin{lem}
If $V$ is an isometry  on $H$, then \label{2.4}
$$M_V(1)=0 \text{ or }1.$$
\end{lem}
\begin{proof}
Since $V$ is an isometry. By the von Neumann-Wold Decomposition theorem \cite[p.112, Theorem 23.7]{Co00}
$$
V=S'\oplus U,
$$
where $S'$ is a unilateral shift and $U$ is a unitary operator.

If $S'\neq0$, then for any unit vector $e\in H\ominus VH$,
$$
M_V^{e}(1)=M_S^{e}(1)=1,
$$
and hence $M_V(1)=1$.

If $V$ is a unitary operator. Then there are two cases under consideration:
\textbf{I}. $M_V^{e}(1)=0 $ for all unit vectors $e\in H$; \textbf{II}. There is
a unit vector $e$ such that   $M_V^{e}(1)\neq 0 .$

 \textbf{Case I.} $M_V^{e}(1)=0 $ for all unit vectors $e\in H$. In this case, $M_V(1)=0$.

 \textbf{Case II.} There is a unit vector $e$ such that   $M_V^{e}(1)\neq 0 .$
By the same reasoning as in the proof of Lemma \ref{2.3}, there is a probability Borel measure $\mu$ on  $\mathbb{T}$
such that there exists a unitary operator $U_0:[e]\rightarrow L^2(\mu)$ satisfying $$N_{\mu}=U_0V_{[e]}U_0^{-1}\ \ \mathrm{and} \ \ U_0 e=1.$$ Then
\begin{eqnarray*}
M_V^{e}(1)&=& M_{V_{[e]}}^{e}(1)=M_{N_{\mu}}^{1}(1)\\
&=&\bigg(\inf_{f \in A_0}\int|1-f|^2 \,\mathrm{d}\mu\bigg)^{\frac{1}{2}}\,\,\,\,\,
\big(\mathrm{d}\mu=h\frac{\mathrm{d}\theta}{2\pi}+\mathrm{d}\mu_s\big)\\
&=&\exp\bigg[\frac{1}{2}\int_{\mathbb{T}}\log |h|\,\frac{\mathrm{d}\theta}{2\pi} \bigg]\neq 0.
\end{eqnarray*}
This shows that $\log |h|\in L^1(\frac{\mathrm{d}\theta}{2\pi})$.

Since $M_T$ is unitarily invariant for $T$, we have
\begin{eqnarray*}
M_{V_{[e]}}(1)=M_{N_{\mu}}(1)&=&\sup_{ ||f||_{L^2(\mathrm{d}\,\mu)}=1}M_{N_{\mu}}^{f}(1)   \\
&=&\sup_{  ||f||_{L^2(\mathrm{d}\,\mu)}=1}\bigg(\inf_{p \in A_0}\int|1-p|^2 \cdot |f|^2\,\mathrm{d}\mu\bigg)^{\frac{1}{2}}  \\
&=&\sup_{ ||f||_{L^2(\mathrm{d}\,\mu)}=1}\exp\bigg[\frac{1}{4\pi}\int_{-\pi}^{\pi}\log |h(\theta)||f(\theta)|^2\,\mathrm{d}\theta\bigg].
\end{eqnarray*}
That is, \begin{equation}
M_{V_{[e]}}(1)=M_{N_{\mu}}(1)=
\sup_{ ||f||_{L^2(\mathrm{d}\,\mu)}=1}\exp\bigg[\frac{1}{4\pi}\int_{-\pi}^{\pi}\log |h(\theta)||f(\theta)|^2\,\mathrm{d}\theta\bigg].  \label{ad2}
\end{equation}
Put $$E=\mathrm{supp } \,\mu_s \cup h^{-1}\{0,+\infty\},$$  and then $E$   %Baire
has Lebesgue measure zero.
Set
\[
f(x)=\left\{
\begin{array}{ll}
0& x\in E,\\
\frac{1}{\sqrt{h(x)}}& x\not\in E.
\end{array}
\right.
\]
Clearly,
$$
\int_{\mathbb{T}}|f|^2\,\mathrm{d}\,\mu=1.
$$ It is easy to see that
$$
M_{N_{\mu}}^{f}(1)=\exp\bigg[\frac{1}{4\pi}\int_{-\pi}^{\pi}\log |h(\theta)||f(\theta)|^2\,\mathrm{d}\theta\bigg]=1.
$$
Then
$$
1=M_{N_{\mu}}(1)=M_{V_{[e]}}(1)\leq M_V(1)\leq1,
$$
forcing $M_V(1)=1.$
The proof is complete.
\end{proof}
  Combining Lemma \ref{2.3} with Lemma \ref{2.4} yields the following.
\begin{prop}
Suppose $V$ is an isometry on $H$. Then \label{2.5}
$$
M_{V}\equiv 0 \text{ or }M_{V}(p)=M(p),\ \ p\in  \mathbb{C}[z].
$$
In particular, if $V$ is a   non-unitary isometry, then $$M_{V}(p)=M(p),\ \ p\in  \mathbb{C}[z].$$
\end{prop}
Now we are ready to give the proof of Theorem \ref{2.6}.

\noindent \textbf{Proof of Theorem \ref{2.6}.} From the classical Sz.-Nagy's dilation theorem \cite[p.7, Theorem 1.1]{Pa}, $T$ has a unitary dilation.
That is, there is a Hilbert space $K$ with $H\subseteq K$ and a unitary operator $U \in B(K)$ such that
$$P_{H}U^n|_H=T^n,\ \ n\geq 1,$$
where $P_H$ is the projection from $K$ to $H$. Since for each $p\in \mathbb{C}[z]$,
$$p(T)e=P_{H}p(U)e,$$ we have $\|p(T)e\|\leq \|p(U)e\|$.
 This, combined with  (\ref{e2.4}), shows that for  any unit vector $e\in  H$,
$$
M_{T}^{e}(p)\leq M_{U}^{e}(p)=M_{U}^{e}(1)\cdot M(p)  ,\, \   p\in \mathbb{C}[z].
$$
Since $M_{U}^{e}(1) \leq \|e\|=1,$ it follows that
\begin{equation}
M_{T}^{e}(p)\leq  M(p),\ \ p\in \mathbb{C}[z]. \label{e2.5}
\end{equation}
Thus
\begin{eqnarray*}
M_{T}(p)=\sup_{\|e\|=1}M_{T}^e(p)\leq M(p),\ \ p\in \mathbb{C}[z].
\end{eqnarray*}
The proof is complete.
$\hfill \square$
\vskip2mm

\begin{rem}
One may compare the inequality  $$
M_{T}(p)\leq M(p)  ,\, \   p\in \mathbb{C}[z]
$$with the well-known von Neumann's inequality \cite[p.7]{Pa}:
$$
\|p(T)\|\leq \|p\|_{\infty} \triangleq \sup\{|p(z)|:|z|=1\}  ,\, \  p\in \mathbb{C}[z],
$$  where $T$ is a contraction in both cases.
\end{rem}
It is not difficult to verify the followings:
$$
M_{T}^{ce}(p)= |c|M_{T}^{e}(p)\text{, c}\in\mathbb{C};
$$
\begin{equation}
  M_{cT}^{e}(p)= M_{T}^{e}(p(cz))\text{,
c}\in\mathbb{C}\setminus\{0\}. \label{e2.6}
\end{equation}
 Combining  (\ref{e2.5}) with (\ref{e2.6}) shows that  for any $T\in B(H)$ and $e\in H$,
\begin{equation}
       M_{T}^{e}(p)    \leq  \|e\|\cdot M\big(p(\|T\|z)\big).\label{e2.7}
\end{equation}
\begin{cor}
Let $T\in B(H)$  be a contraction. Then for any unit vector $e\in H$,\label{cor 2.8}
$$
M_{T}^{e}(p)\leq  M(p), \ \ p\in \mathbb{C}[z].
$$  Moreover, the equality  holds for all $p\in \mathbb{C}[z]$ if and only if
$\{T^n e\}_{n=0}^{\infty}$ is an orthonormal sequence.
\end{cor}
\begin{proof}
The inequality follows from (\ref{e2.5}).

Suppose that $M_{T}^{e}(p)=  M(p) $   holds for all $p\in \mathbb{C}[z]$. We will show that
$\{T^n e\}_{n=0}^{\infty}$ is an orthonormal sequence. For this, notice that
$$
M_{T}^{e}(z^n)=M(z^n)=1=\text{dist}\big(T^n e,\overline{\text{\upshape span}}\{T^{n+1}e,T^{n+2}e,\cdots \}\big),\ n\geq 0.
$$
Since
$$\|T^n e\|\leq 1\text{  and  } \text{dist}\big(T^n e,\overline{\text{\upshape span}}\{T^{n+1}e,T^{n+2}e,\cdots \}\big)=1,$$
we have
$$ \|T^n e\| =  1 \ \ \mathrm{and} \ \
  T^n e  \perp  \overline{\text{\upshape span}}\{T^{n+1}e,T^{n+2}e,\cdots \}.$$

On the other hand, if $\{T^n e\}_{n=0}^{\infty}$ is an orthonormal sequence, and write
  $$[e]= \overline{\text{\upshape span}}\{e,Te,\cdots,T^n e,\cdots \}.$$
 Then the restriction  $T|_{[e]}$ of $T$ on $[e]$ is an isometric operator, and $$M_{T\mid_{ [e]}}^{e}(1)=\text{dist}(e,[Te])=1.$$
  Then by (\ref{e2.4}), we have
$$
M_{T}^{e}(p)=M_{T\mid_{[e]}}^{e}(p)=M_{T\mid _{[e]}}^{e}(1)\cdot M(p)=M(p), \ \ p\in \mathbb{C}[z].
$$
This completes the proof.
\end{proof}

\subsection{Multiplicativity}

Let  $\Phi$ be a map from $\mathbb{C}[z]$ to $\mathbb{R}$. If $\Phi(pq)=\Phi(p)\Phi(q)$ for all
$p,q\in\mathbb{C}[z]$,  then $\Phi$ is called \emph{multiplicative}. Clearly, Mahler's measure is multiplicative.
The remaining part of this section  focuses on the multiplicativity properties of  the $T$-Mahler measure.

Recall that the  Bergman space $L_a^2 (\mathbb{D})$ is defined by %\cite{DS}
$$
L_a^2 (\mathbb{D})=\{f:f\text{ is holomorphic on } \mathbb{D}\text{ such that }\linebreak
\int_{\mathbb{D}}|f(z)|^{2}\mathrm{d}A(z)<+\infty\},
$$
where $\mathrm{d}A(z)=\frac{\mathrm{d}x\mathrm{d}y}{\pi}
$ is the normalized area measure on $\mathbb{D}$. Denote by $B$ the Bergman shift, defined by
\begin{eqnarray*}
(Bf)(z)=zf(z),\ \ f\in L_a^2 (\mathbb{D}).
\end{eqnarray*}
Then we have $ M_{B}^{1}(z^n)=\|z^n\|=\frac{1}{\sqrt{n+1}}$, which implies that  $M_{B}^{1}$ is not multiplicative.
In general, we have the following lemma.
\begin{lem}Let  $\{e_n\}_{n\geq1}$ be an orthonormal basis of $H$ and
suppose that $T\in B(H)$ is a weighted shift such that\label{lem 2.9}
$$
Te_n=a_n e_{n+1},\ a_n\neq 0 \text{, } n\geq 1.
$$
  Then $M_T^{e_1}$ is multiplicative if and only if $|a_n|=|a_1|$ for all $n\geq1$, and in this case $M_T^{e_1}(p)=M(p(|a_1| z))$.
\end{lem}
\begin{proof}

First assume that $M_T^{e_1}$ is multiplicative. Then it is easy to see that $M_T^{e_1}(z^n)=\prod_{i=1}^{n}|a_i|$. Thus
$$
\prod_{i=1}^{n}|a_i|=|a_1|^n\ \text{, }\ n\in \mathbb{N}.
$$
Therefore $|a_n|=|a_1|$, $n\geq 1$.

Now assume that $|a_n|=|a_1|$, $n\geq 1$.  Then we can write $$T=|a_1|US,$$ where $U$ is a  unitary operator satisfying
 $U e_n=e^{i\theta_n}e_n,\theta_n\in \mathbb{R}$, $n\geq1$ and $S$ is the Hardy shift.

Observe that $$\|p(c US)e_1\|=\|p(c S)e_1\|$$ holds for all $c\in  \mathbb{R}$ and $p\in \mathbb{C}[z]$.
In fact, this identity is trivial if $p$ is monomial; and in  general, it follows from the   orthogonality of $\{e_n: n\geq 1\}$.

 Therefore,  for each $p\in \mathbb{C}[z]$ we have
\begin{eqnarray*}
M_T^{e_1}(p)&=&\inf\{\|[I-q(T)T]p(T)e_1\|:q\in \mathbb{C}[z]\}\\
&=&\inf\{\|[I-q(|a_1|S)|a_1|S]p(|a_1|S)e_1\|:q\in \mathbb{C}[z]\}\\
&=&\inf\{\|[I-q(S)S]p(|a_1|S)e_1\|:q\in \mathbb{C}[z]\}(\text{ since } |a_1|\neq 0)\\
&=&M_S^{e_1}(p(|a_1|z))\\
&=&M(p(|a_1|z)).
\end{eqnarray*}
 The proof is complete.
\end{proof}

If one replaces  $M_T^{e_1}$  with $M_T$ in Lemma \ref{lem 2.9}, then we get a similar result.
\begin{prop}
With the same assumption as in Lemma 2.9 and assume that  $\{|a_n|\}_{n=1}^{\infty}$
is a decreasing sequence. Then $M_T$ is multiplicative if and only if
$|a_1|=|a_n|$, $n\geq1$. In this case,  $M_T(p)=M(p(|a_1| z))$, $p\in\mathbb{C}[z]$.
\end{prop}
\begin{proof}
If $|a_1|=|a_n|$ for all $n\geq1$, then by (\ref{e2.7})
$$
M_T(p)\leq M(p(|a_1|z)).
$$
By Lemma \ref{lem 2.9}, we have
$$
M_T^{e_1}(p)=M\big(p(|a_1| z)\big).
$$
Thus
$$
M_T(p)= M\big(p(|a_1|z)\big).
$$
This implies that $M_T$ is multiplicative.

On the other hand, assume that $M_T$ is multiplicative. %Then we claim that $$M_T(z^n)=\prod_{i=1}^{n}|a_i|.$$
Since
$$
M_T^{e_1}(z^n)=\text{dist}(T^n e_1,[T^{n+1} e_1])=\|T^n e_1\|=\prod_{i=1}^{n}|a_i|,
$$
and for any unit vector $e\in H$,
$$
M_T^{e}(z^n)=\text{dist}(T^n e,[T^{n+1} e])\leq \|T^n e\|\leq \sup_{m\geq1}\prod_{i=m}^{n+m-1}|a_i|=\prod_{i=1}^{n}|a_i|,
$$
we get
$$
M_T(z^n)=\prod_{i=1}^{n}|a_i|.
$$
Since $M_T$ is multiplicative, we have
$$\prod_{i=1}^{n}|a_i|=|a_1|^n, n\geq1.$$
By induction, we get
$$
|a_1|=|a_n| \text{, }n\geq1,
$$
as desired. This completes the proof.
\end{proof}

\begin{cor}
Let $T\in B(H). $ If there is a unit vector $e\in H$ such that $M_{T}^{e}(z)=\|T\|\neq 0$, then $M_{T}^{e}$ is multiplicative if and only if
$$M_{T}^{e}(p)=M(p(\|T\|z)),\ \ p\in \mathbb{C}[z].$$
\end{cor}
\begin{proof}
The sufficiency is trivial.

 Suppose that $M_{T}^{e}$ is multiplicative. Then it is not diffcult to verify that
$M_{\frac{T}{\|T\|}}^{e}$ is multiplicative and
$$
 M_{\frac{T}{\|T\|}}^e(z)=M_{T}^e(\frac{z}{\|T\|})=\frac{M_{T}^e(z)}{\|T\|}=1.
$$
Thus $$M_{\frac{T}{\|T\|}}^{e}(z^n)=1, \ \ n\geq 0.$$ By the proof of Corollary \ref{cor 2.8},
one  sees that $\big\{(\frac{T}{\|T\|})^n e\big\}_{n=0}^{\infty}$ is an orthonormal sequence.
Thus
$$M_{\frac{T}{\|T\|}}^{e}(p)=M(p), p\in \mathbb{C}[z].$$
Then by  (\ref{e2.6}), we have
$$
M_{T}^{e}(p)=M_{\|T\|\cdot \frac{T}{\|T\|}}^{e}(p)=M_{ \frac{T}{\|T\|}}^{e}(p(\|T\|z))=M(p(\|T\|z)).
$$
 The proof is complete.
\end{proof}

\section{ Subharmonic Operators}
~~~~In order to generalize Lehmer's problem in the context of operator theory, we will introduce and study an important class of operators, the so-called  subharmonic operators which is closely related to the operator-theoretic Mahler measure.

First, the definition of  subharmonic operators is given as follows:
\begin{Def}
For an operator $T \in B(H)$, if there is a unit vector $e\in H$ such that
$$
\|p(T)e\|\geq |p(0)|, \ \ p\in\mathbb{C}[z],
$$
then   $T$ is called to be subharmonic on $e$.

$T$ is called subharmonic  if for any $\epsilon>0$ there is a unit vector $e\in H$ such that
$$
\|p(T)e\|\geq |p(0)|(1-\epsilon), \ \  p\in\mathbb{C}[z].
$$
\end{Def}
Let us see a simple example of subharmonic operators.
\begin{exam}As mentioned before, $S$ denotes the Hardy shift. For any  $p\in\mathbb{C}[z]$,
$$
||p(S)1||=||p||_{H^2}=\bigg(\frac{1}{2\pi}\int_{-\pi}^{\pi}|p^2(e^{i\theta})|\,\mathrm{d}\theta\bigg)^{\frac{1}{2}}.
$$
Since $|p^2(z)|$ is a subharmonic function on $\mathbb{C}$, we have
$$
\bigg(\frac{1}{2\pi}\int_{-\pi}^{\pi}|p^2(e^{i\theta})|\,\mathrm{d}\theta\bigg)^{\frac{1}{2}}\geq|p(0)|.
$$
Then
$$
||p(S)1||\geq |p(0)|\text{, } p\in\mathbb{C}[z].
$$
Hence $S $ is  subharmonic on $1$. Similarly, the Bergman shift $B$ is also a subharmonic operator.
\end{exam}

By definition, it is easy to see that if $T$ is subharmonic, then \begin{equation}
M_T(p)\geq|p(0)|  ,\ \ p\in \mathbb{C}[z].       \label{miss}
\end{equation}
 Similarly, if $T$ is subharmonic on some unit vector $e$,
 then \begin{equation}
 M_T^e(p)\geq|p(0)|  ,\ \ p\in \mathbb{C}[z].      \label{miss2}
\end{equation}

By (\ref{q1.1}), the original Lehmer's problem is equivalent to the question:
\vskip2mm \emph{Is there a sequence of non-cyclotomic
integral polynomial $p_n$ satisfying  $p_n(0)\neq0$ such that}
$$
\lim_{n\to \infty}M_S(p_n)=1 \  (\mathrm{or} \ \lim_{n\to \infty}M_S^1(p_n)=1)?
$$
Inspired by this, if $T$ is subharmonic (or $T$ is subharmonic on some unit vector $e$). Then we raise the following question for $M_T$ ( or $M_{T}^{e}$):
\vskip2mm \emph{Is there a sequence of non-cyclotomic
integral polynomial $p_n$ satisfying  $p_n(0)\neq0$ such that}
\begin{equation}
\lim_{n\to \infty}M_T(p_n)=1 \  (\mathrm{or} \ \lim_{n\to \infty}M_T^e(p_n)=1)?\label{q3.3}
\end{equation}
This is the operator version of Lehmer's problem and it will be answered under some mild condition in section 4.

\subsection{Properties of subharmonic operators}
Furthermore, the following theorem describes subharmonic operators, and its proof is placed at the end of section 3.2.
\begin{thm}
Suppose $T\in B(H)$, then the following statements are equivalent.      \label{t3.2}
\begin{itemize}
\item[$(1)$]  $M_T(1)=1$.
\item[$(2)$] $T$ is subharmonic.
\item[$(3)$] $T$ is subharmonic on some unit vector $e$.
\item[$(4)$] There is a unit vector $e\in H$ such that $e \perp\overline{\text{\upshape span}}\{Te,T^2 e,\ldots\}$.
\end{itemize}
\end{thm}
The following are some other examples of subharmonic operators.
\begin{lem}
Suppose $T\in B(H)$. Then the followings hold:       \label{3.2}
\begin{itemize}
\item[$(1)$]   If $\ker T\neq\{0\},$ or $ \overline{ran T}\neq H$, then $T$ is subharmonic;
\item[$(2)$]   All weighted shift operators are subharmonic;
\item[$(3)$]  All semi-Fredholm operators with nonzero index are subharmonic.
\end{itemize}
\end{lem}
\begin{proof}

(1).  Assume that either $\ker T\neq\{0\} $ or $ \overline{ranT}\neq H$.
Then pick a unit vector $e $ such that  $e\in \ker\, T$ or  $e\in(\overline{ranT})^{\perp}$, and in either case we have
 $$e \perp\overline{\text{\upshape span}}\{Te,T^2 e,\ldots\}.$$Then it follows that $T$ is subharmonic on $e$.

Both (2) and (3) follow directly from (1). The proof is
 complete.
\end{proof}
By Lemma \ref{3.2}, one  sees that many analytic multiplication operators on function spaces are subharmonic. For
example,  the Hardy shift $S$, the Dirichlet shift \cite{ARSW} and the Bergman shift $B$.

Applying Theorem \ref{t3.2} and Lemma \ref{3.2}, and using a matrix decomposition technique we get the following:
\begin{cor}      \label{l3.6}
If $T\in M_n(\mathbb{C} )$, then $T$ %B(^n) $
 is subharmonic if and only if \linebreak $\ker   T\neq\{0\}$.
\end{cor}
\begin{proof}
If $\ker T\neq\{0\}$, then by  Lemma \ref{3.2}(1)   $T$ is subharmonic.

Suppose $T$ is subharmonic. By Theorem \ref{t3.2}(3), there is a unit vector $e$ such that
 $T$ is subharmonic on $e$. Write
 $$[ e] =\text{\upshape span}\{e,Te,T^2 e,\ldots\}\text{ and }T_1=T|_{[ e]}.$$
 Then   decompose $T$ as
\[
T=\left[
\begin{array}{cc}
T_1&T_2\\
0&T_3
\end{array}\right]
\]
with respect to $[ e] \oplus[ e]^{\perp}.$

Assume conversely that $  \ker T=\{0\}$, then $T$ is invertible, and so is $T_1$.
However, by  Theorem \ref{t3.2}(4), one can decompose $T_1$ as follows:
\[
T_1=\left[
\begin{array}{cc}
0&0\\
T_{1,1}&T_{1,2}
\end{array}\right]
\]
corresponding to $[e]=\text{\upshape span}\{e\}\oplus \text{\upshape span}\{Te,T^2 e,\ldots\}$. This is a contradiction.
% to that $T_1$ is invertible.
The proof is complete.
\end{proof}

 By the proof of Corollary \ref{l3.6}, one can see that: if  $T\in B(H)$ and $T^*=T$, then
  $T$ is subharmonic if and only if $\ker T\neq\{0\}.$

\subsection{The quantity $E(T)$}

Recall that for $T\in B(H)$,
$$M_T(1) =\sup_{||e||=1}\inf_{q\in\mathbb{C}[z] }||[I-q(T)T]e||=\sup_{||e||=1}\text{dist}(e,[Te]).$$
We write $E(T)=M_T(1)$. The quantity $E(T)$ carries key information of $T$ and it will play an important role in the proof of Theorem \ref{t3.2}.

The following inequality establish an interesting connection between the $T$-Mahler measure and $E(T)$.
\begin{prop}
For each $T\in B(H),$ we have                \label{3.3}
$$M_T(p)\leq \|p(T)\|\cdot E(T) ,\ \ p\in \mathbb{C}[z].$$
In particular, if T is a contraction, then
$$M_T(p)\leq \|p\|_{\infty}\cdot E(T).$$
\end{prop}
\begin{proof}
By the definition of $M_T(p)$, we have
\begin{eqnarray*}
M_T(p)&=&\sup_{\|e\|=1} \inf_{q\in \mathbb{C}[z]}\|[I-q(T)T]p(T)e\|\\
      &=&\sup_{\|e\|=1} \bigg\{\|p(T)e\|\cdot \inf_{q\in \mathbb{C}[z]}\|[I-q(T)T]\frac{p(T)e}{\|p(T)e\|}\|\bigg\}\\
      &\leq&\sup_{\|e\|=1} \|p(T)e\|\cdot M_T(1)\\
      &=&\|p(T)\|\cdot E(T).
\end{eqnarray*}
If $T$ is a contraction, then by von Neumann's inequality $\|p(T)\|\leq \|p\|_{\infty}$, and hence
$$M_T(p)\leq \|p\|_{\infty}\cdot E(T).$$
The proof is complete.
\end{proof}
As a consequence of Proposition \ref{3.3}, if $E(T)=0$, then $M_T\equiv 0$. Also,
$E(T)$ is a unitary invariant for $T$. %Actually, $M_T(1)$ is an interesting quantity for $T$.

Lemma \ref{2.4} says that if $T$ is an isometry, then $E(T)=0$ or $1$.
The following gives a  generalization of this result, with  a shorter proof.
\begin{prop}
For each $T\in B(H)$, $E(T)\in \{0,1\}$.           \label{3.4}
\end{prop}
\begin{proof}
Suppose $M_T(1)\neq 0$, and we will show that  $M_T(1)= 1$ to finish the proof.

Since
$M_T(1)\neq 0$, there is a unit vector $e\in H$ such that
$$\text{dist}(e,[Te])\neq 0.$$
Write
$H_0=\overline{\text{\upshape span}}\{e,[Te]\}$, and put $T_0=T|_{H_0}$.

Since
$$\overline{\text{ran}T_0}=[Te]\varsubsetneqq H_0,$$
there is a unit vector $e_0\in H_0\ominus[Te]$.

Since
 $[T_0 e_0]\subseteq [Te] $, we get  $$ e_0 \perp [T_0 e_0].$$
Hence dist$(e_0,[T_0e_0])=1$, forcing $M_{T_0}(1)= 1$. On the other hand, it is easy to see that $$ M_{T_0}(1)\leq M_T(1)\leq 1.$$
Therefore,  $M_{T}(1)=1$,  completing the proof.
\end{proof}

We now will give the proof of Theorem \ref{t3.2}.

\noindent \textbf{Proof of Theorem \ref{t3.2}.}

$(2) \Rightarrow (1)$. If $T$ is subharmonic, then for any $\epsilon>0$ there is a unit vector $e$ such that
$$
\|p(T)e\|\geq |p(0)|\cdot(1-\epsilon), \ \ p\in\mathbb{C}[z].
$$
In particular, put  $p(z)=1-zq(z)$, and then
$$
\|[I-q(T)T]e\|\geq 1-\epsilon , \ \ q\in\mathbb{C}[z].
$$
This means that
$$
 1-\epsilon \leq M_T^e(1)\leq M_T(1)\leq 1.
$$
By    arbitrariness of $\epsilon$, we have $M_T(1)= 1$.

$(1 )\Rightarrow (2)$. Suppose that $M_T(1)=1$. Then for any $\epsilon>0$ there is a unit vector $e$ such that
$$1-\epsilon\leq M_T^e(1)\leq1,$$
which implies that
$$
\|[I-q(T)T]e\|\geq 1-\epsilon, \  \ q\in\mathbb{C}[z].
$$
Therefore
$$
\|p(T)e\|\geq |p(0)|\cdot(1-\epsilon),\ \ p\in\mathbb{C}[z].
$$

$(4) \Rightarrow(3)$ is trivial.

$(3)\Rightarrow (4)$. If $T$ is subharmonic on unit vector $e$, then $$\|p(T)e\|\geq |p(0)| , \ \ p\in\mathbb{C}[z].$$
In particular, for any $q\in\mathbb{C}[z]$, we have
$$
\|[I-q(T)T]e\|\geq 1 ,
$$
which gives dist$(e,[Te])=1$.
Thus,  $e \perp\overline{\text{\upshape span}}\{Te,T^2 e,\ldots\}$.
\vskip2mm
$(3) \Rightarrow(2)$ is trivial.

$(2) \Rightarrow(3)$. If $T\in B(H)$ is subharmonic. Then $M_T(1)=1$. By the proof of Proposition \ref{3.4}, there is a unit vector $e\in H$ such that
$$e \perp\overline{\text{\upshape span}}\{Te,T^2 e,\ldots\}.$$
Thus $T$ is subharmonic on $e$.

 The proof of the theorem is complete.
$\hfill \square$
\vskip2mm

\section{Applications}
~~~~In this section, we will apply the results of previous sections to the operator version of Lehmer's
problem (\ref{q3.3}).
%Its proof is placed at the end of this section.

The following is our main result in this section, which answers (\ref{q3.3}) under some mild condition. As applications, we give new proofs of some known results.
\begin{thm} Suppose $T\in B(H)$ is subharmonic on some unit vector $e$,  contractive and
$T^n\stackrel{SOT}{\rightarrow}0$.
Then there is a sequence of non-cyclotomic integral polynomials $\{p_n\} $ with $p_n(0)\neq0$, such that \label{3.9}
$$
\lim_{n\to \infty}M_{T}^{e}(p_n)=1.
$$
\end{thm}
The proof of Theorem \ref{3.9} is given after Remark \ref{3.8}.
\begin{cor}
Suppose $T\in B(H)$ is subharmonic and $\|T\|<1$. \label{3.10}
Then there is a sequence of non-cyclotomic polynomials $\{p_n\}\in \mathbb{Z}[z]$ with \linebreak $p_n(0)\neq0$, such that
$$
\lim_{n\to \infty}M_{T}(p_n)=1.
$$
\end{cor}
\begin{proof}
Combine (\ref{e2.6}), (\ref{miss}) and Theorem \ref{3.9}.
\end{proof}

To prove Theorem \ref{3.9} we need the following proposition.
\begin{prop}
Let $T\in B(H)$ be a subharmonic contraction. Then  for any cyclotomic  polynomial $p$,\label{3.7}
$$M_{T}(p)=1.$$
\end{prop}
\begin{proof}
 Assume that $T\in B(H)$ is a subharmonic operator with $\|T\|\leq1$. Then by Theorem \ref{2.6}  we have
$$
M_{T}(q)\leq  M(q), \ \ q\in\mathbb{C}[z].
$$
For each cyclotomic polynomial $p$, we have
$$
M_{T}(p)\leq M(p)=1.
$$
Since $p$ is a cyclotomic polynomial, then $|p(0)|=1$. By (\ref{miss}), we have
$$
M_{T}(p)\geq |p(0)|=1.
$$
Therefore $M_{T}(p)=1$, as desired. The proof is complete.
\end{proof}
\begin{rem}
Similarly,   one can show that if $T$ is a contraction and $T$ is subharmonic on some unit vector $e$,
then  $M_{T}^{e} (p)=1$
for any cyclotomic  polynomial $p$.
Indeed, it is a nontrivial result, because on Bergman space $L_a^2 (\mathbb{D})$ it is hard to compute \label{3.8}
$M_{B}^{1}(z+1)$ via definition.
\end{rem}
Now we come to the proof of Theorem \ref{3.9}.

\noindent \textbf{Proof of Theorem \ref{3.9}.}
Write $p_n=z^n+z+1(n>2)$, which is a sequence of non-cyclotomic polynomials [EW, p.78, Exercise 3.12].
Since $T$ is subharmonic on $e$ and $\|T\|\leq 1$, by Remark \ref{3.8} we have
$$
M_{T}^{e}(z+1)=1.
$$
Then by definition of $M_{T}^{e}$, for any $\epsilon >0$, there is a polynomial $q$ such that
$$
\|[I-q(T)T ](T+1)e \|<1+\epsilon.
$$
Since $
\lim_{n\to\infty}\|T^n e\|=0,
$ we have
\begin{eqnarray*}
\lim_{n\to \infty}\|[I-q(T)T ](T^n+T+1)e \|&=&\|[I-q(T)T](T+1)e \|\\
&<&1+\epsilon.
\end{eqnarray*}
Then there is a natural number $N$ such that
$$
\|[I-q(T)T](T^n+T+1)e \|<1+\epsilon, \ \ n\geq N.
$$
By (\ref{miss2}), we have $M_{T}^{e}(p)\geq|p(0)|$ for all $p\in \mathbb{C}[z]$, and hence
$$
1\leq M_{T}^{e}(z^n+z+1)< 1+\epsilon,\ \ n\geq N.
$$
By   arbitrariness of $\epsilon$, we have
$$
\lim_{n\to \infty}M_{T}^{e}(z^n+z+1)=1,
$$
 completing the proof.
$\hfill \square$
\vskip2mm

As an application of  Theorem \ref{3.9}, we have the following example. %(????The following example shows what?where is the "application ")
\begin{exam}
Recently, {\upshape I.Pritsker \cite{Pr08a}} defined an areal analog of Mahler's measure as the following: For each polynomial $p\in \mathbb{C}[z]$
$$
\|p\|_{0}\triangleq \exp\bigg\{\int_{\mathbb{D}}\log|p(z)|\mathrm{d}A(z)\bigg\}.\label{e4.5}
$$
This can be regarded as a Bergman-space version of Mahler's measure.

Notice that for each $p\in \mathbb{Z}[z]$ with $p(0)\neq 0$,   $\log|p(z)|$ is subharmonic on $\mathbb{D}$ and $\log |p(0)|\geq 0$, and then
$ ||p||_{0}\geq 1$. Later, below Proposition \ref{3.14} we will apply  Theorem \ref{3.9} to show that $$
\lim_{n\to \infty}||p_n||_{0}=1.
$$
\end{exam}
We have the following relations on $||p||_{0}$, $ M_{B}^{1}(p)$ and $ M(p)$. %which says that B-Mahler's measure on 1 is closer to Mahler's measure than the
% areal analog of Mahler's measure.
\begin{prop}
On Bergman space $L_a^2 (\mathbb{D})$,
$$
||p||_{0} \leq M_{B}^{1}(p)\leq M(p)\text{, } p\in\mathbb{C}[z],\label{p4.6}
$$ where $B$ is the Bergman shift.\label{3.14}
\end{prop}
\begin{proof}
Observe that $B$ is subharmonic on the unit vector $1$
and $\|B\|\leq 1$.
Then by Corollary 2.8,
$$ M_{B}^{1}(p)\leq M(p)\text{, } p\in\mathbb{C}[z].$$

It remains to show that
$$||p||_{0} \leq M_{B}^{1}(p) , \ \  p\in \mathbb{C}[z].$$
For this, notice that
$$
M_{B}^{1}(p)=\bigg(\inf_{q\in\mathbb{C}[z],q(0)=0}\int_{\mathbb{D}}|[1-q(z)]p(z) |^2\mathrm{d}A(z)\bigg)^{\frac{1}{2}}.
$$
Since $\mathrm{d}A(z) $
 is a probability measure,   then by Jensen's Inequality [Ru, p.62, Theorem 3.3], for each $q\in\mathbb{C}[z]$ with $q(0)=0$ we have
 \begin{eqnarray*}
&&\int_{\mathbb{D}}|[1-q(z)]p(z)|^2\mathrm{d}A(z)\\
&\geq & \exp\bigg\{\int_{\mathbb{D}}\log|[1-q(z)]p(z)|^2\mathrm{d}A(z)\bigg\}\\
&=&\exp\bigg\{\int_{\mathbb{D}}\log|1-q(z)|^2\mathrm{d}A(z)+\int_{\mathbb{D}}\log|p(z) |^2\mathrm{d}A(z)\bigg\}.
\end{eqnarray*}
Since the function $\log|1-q(z)|^2$ is subharmonic on $\mathbb{D}$ and $\log|1-q(0)|^2=0$, % (see [Ru, pp. 335-337] about subharmonic functions),
we have
\begin{eqnarray*}
&&\exp\bigg\{\int_{\mathbb{D}}\log|1-q(z)|^2\mathrm{d}A(z)+\int_{\mathbb{D}}\log|p(z) |^2\mathrm{d}A(z)\bigg\}\\
&\geq &\exp\bigg\{\int_{\mathbb{D}}\log|p(z)|^2\mathrm{d}A(z)\bigg\}\\
&=&||p||_0^2.
\end{eqnarray*}
Thus
$$||p||_{0} \leq M_{B}^{1}(p),\ \ p\in \mathbb{C}[z]$$ as desired. This completes the proof.
 \end{proof}

Observe that the Bergman shift $B$ is subharmonic on the unit vector 1,
$\|B\|\leq 1$ and $B^n\stackrel{SOT}{\rightarrow}0$. Write $p_n(z)=z^n+z+1$. Then by Proposition \ref{3.14},
$$
1\leq ||p_n||_{0} \leq M_{B}^{1}(p_n).
$$
Applying  Theorem \ref{3.9} gives that $$
\lim_{n\to \infty}M_{B}^{1}(p_n)=1.
$$
Thus we obtain
$$
\lim_{n\to \infty}||p_n||_{0}=1.
$$
\begin{rem}Consider the weighted Bergman space\label{r4.7}
$$
L_a^2 (\mathbb{D},\rho(|z|^2)\mathrm{d}A(z))
=\bigg\{f: f \text{is analytic on } \mathbb{D}  , \int_{\mathbb{D}}|f(z)|^2\rho(|z|^2)\,\mathrm{d}A(z)<\infty\bigg\}.
$$
 In {\upshape \cite[Example 1]{Hu}}, the quantity $\|p\|_\rho$ is defined by
$$
\|p\|_\rho \triangleq\exp\bigg(\int_{\mathbb{D}}\log|p(z)|\rho(|z|^2)\,\mathrm{d}A(z)\bigg).
$$
 Then by the same reasoning as above, % in the proof of the above argument,
one gets
$$
\lim_{n\to \infty}||p_n||_{\rho}=1.
$$
This result was first obtained in {\upshape \cite[Corollary 4]{Hu}}. However, the proof presented here is an operator-theoretic approach,
 which is quite different from that of {\upshape \cite[Corollary 4]{Hu}}.
\end{rem}

%%%%%%%%%%%%%%%%%%%%%%%%%%%%%%%%%%%%%%%%%%%%%%%%%%%%%%%%%%%%%%%%%%%%%%%%%
% body of paper
%%%%%%%%%%%%%%%%%%%%%%%%%%%%%%%%%%%%%%%%%%%%%%%%%%%%%%%%%%%%%%%%%%%%%%%%%

\textbf{Acknowledgements:}   The authors thank the referee for helpful suggestions which make this paper more readable. We also wish to express our thanks to Professor X.~Fang  for his interest and suggestions.
%%%%%%%%%%%%%%%%
% bibliography
%%%%%%%%%%%%%%%

% Set bibliography items using the "thebibliography" environment  and following
% the style used by the AMS journals.
%
% If the bibliography is generated by a bibtex database, use "amsplain" or
% "amsalpha" as bibliography style

\end{document}